\documentclass[12pt, titlepage]{article}
\usepackage{eurosym}
\usepackage{amsmath,amssymb,amsthm,amsfonts}
\usepackage{epstopdf}
\usepackage{subcaption}
\usepackage{appendix}
\usepackage{amsmath}
\usepackage{amsfonts}
\usepackage{amssymb}
\usepackage{multirow}
\usepackage{rotating}
\usepackage{xcolor}
\usepackage{lscape}
\usepackage{verbatim}
\usepackage{setspace}
\usepackage{cancel}
\usepackage{accents}
\usepackage{natbib}
\usepackage{apalike}

\numberwithin{equation}{section} 

\newtheorem{theorem}{Theorem}[section]

\newtheorem{proposition}{Proposition}[section]
\newtheorem{definition}{Definition}[section]
\newtheorem{remark}{Remark}
\newtheorem{example}{Example}

\def\XX{\boldsymbol{X}}
\def\xx{\boldsymbol{x}}

\def\rr{\boldsymbol{r}}
\def\ee{\boldsymbol{e}}

\def\DDD{\mathcal{D}}

\def\RRR{\boldsymbol{R}}
\def\PP{\mathbb{P}}

\def\ss{\boldsymbol{s}}

\def\FFF{\mathcal{F}}

\def\YY{\boldsymbol{Y}}

\def\mm{\boldsymbol{m}}

\def\RR{\mathbb{R}}

\def\QQ{\mathbb{Q}}
\def\ff{\boldsymbol{f}}

\def\design{\mathcal{X}}
\def\designk{\mathcal{X}_k}

\def\DDD{\mathcal{D}_3}
\def\VaR{\text{VaR}}
\def\Cov{\text{Cov}}

\begin{document}
\onehalfspacing
\author{ R.FONTANA \\ \textit{\ 
Department of Mathematical Sciences G. Lagrange,} \\ {Politecnico di
Torino.}
\\ P. SEMERARO \\ \textit{\ 
Department of Mathematical Sciences G. Lagrange,} \\ {Politecnico di
Torino.}}
\title{Geometrical representation and dependence structure of three-dimensional Bernoulli distributions}
\maketitle

\begin{abstract}
This paper fully characterizes the geometrical structure of the class of distributions of three-dimensional Bernoulli random variables with equal means, $p$. We  find all the geometrical generators in closed form as functions of $p$. This result stems from an algebraic representation of the class that encodes the statistical properties of Bernoulli distributions. We  study extremal negative dependence within the class and provide an application example by finding the impact of negative dependence to minimal aggregate risk. The application relies on a game theory approach.
\noindent 
\hspace{1cm} 
\end{abstract}

\section*{Introduction}\label{sec.1}
Multivariate Bernoulli distributions are significant in various fields, such as clinical trials and finance. An open issue in applied probability is the characterization of extremal negative dependence within the Fr\'echet class of distributions with given marginals, which, in the case of Bernoulli distributions, corresponds to specified means. While extremal positive dependence is achieved by the upper Fr\'echet bound, the lower Fr\'echet bound -- that in dimension two corresponds to the extremal negative dependence   -- fails to be a distribution in dimension higher than two. Therefore characterizing negative dependence in dimension higher than two is more challenging, with only partial results available (see, e.g., \cite{puccetti2015extremal}).
\cite{fontana2018representation} and  \cite{fontana2024high} provide a geometrical  and an algebraic representation of the class $\FFF_d(p)$ of $d$-dimensional Bernoulli probability mass functions (pmf) with common marginal mean $p$.  Both representations are powerful tools for studying the statistical properties of this class. In particular, the geometrical structure of  $\FFF_d(p)$ is a convex polytope. The extremal points of the convex polytope -that belong to $\FFF_d(p)$ and therefore are pmfs themselves- play a key role in studying the statistical properties of the class. Indeed,  relevant quantities such as moments or convex functionals reach their maximum and minimum value on the extremal points, as discussed in \cite{fontana2018representation}, \cite{fontana2021exchangeable}, and \cite{fontana2022computational}. For this reason we refer to  the dependence structure of the extremal points as to extremal dependence.
Given $p$, we can find the extremal points using a geometrical software, 4ti2. Actually, this is feasible in low dimension since their number increases very fast. This is why \cite{fontana2024high} introduced the algebraic representation: it provides a way to find -at least in principle- the extremal points using the algebraic generators that are  known analytically. Nevertheless, the algebraic structure reveals to encode statistical properties and to be a powerful tool also in low dimension. 

The main result of this paper is to analytically find all the extremal points as functions of $p$ in the three-dimensional case, $\FFF_3(p)$. We  provide a classification of the extremal points of $\FFF_3(p)$, mainly based on their algebraic properties,  and study their correlation  structure. We then focus on the characterization of extremal negative dependence, that, as discussed above,  is not a trivial concept in dimension higher than two. We can find different extensions of extremal negative dependence to dimensions higher than two. Here, we consider a notion introduced by \cite{puccetti2015extremal} called $\Sigma$-countermonotonicity. The $\Sigma$-countermonotonicity property 
is not always a property of a single distribution in dimensions higher than two. In dimension three, if $p\in (1/3, 2/3)$,  it is a property that characterize an entire convex polytope included in $\FFF_3(p),$ that we name $\FFF^{\Sigma}_3(p).$ We explicitly find its extremal points.

As an application example we consider the impact of negative dependence on aggregate risk. Indeed, all the pmfs in this polytope minimize the aggregate risk, meaning that if $\XX\in\FFF^{\Sigma}_3(p)$, then the discrete random variable $S=X_1+X_2+X_3$ is minimal in convex order and, in particular, it has the minimal variance $V(S)$   in the class of discrete distributions on $\{0,1,2,3\}$ with mean $3p$. Despite  $V(S)$ is constant for any $\XX\sim f\in \FFF^{\Sigma}_3(p)$, some dependence features of $\XX$ changes across the class, as pairwise correlation. It is therefore interesting to understand the impact of different correlation structures on the variance of the sum. A way to do that is to find the marginal contributions of the components $X_i, \, i=1,2,3$ to the variance of $S$ and investigate their dependence on the correlation structure. Using a game theory approach and the results in \cite{colini2018variance} we analytically find the marginal contributions to $V(S)$ of the components of any vector $\XX\sim f\in \FFF^{\Sigma}_3(p)$. 

The paper is organized as follows. Section \ref{threedim} recalls the geometrical and algebraic structure of $\FFF_3(p)$. Our main result, the analytical form of  all the extremal pmfs is provided in Section \ref{sec:gen3}.  Section \ref{ExtrCorr} characterizes  extremal negative dependence. The application to the study of the marginal contribution of different negative dependence structures to minimal risk is presented in Section \ref{games}.

\section{Three dimensional Bernoulli distributions}\label{threedim}
This section recalls the geometrical and algebraic representations of three dimensional Bernoulli distributions. The geometrical representation has been introduced in \cite{fontana2018representation}, while the algebraic structure in \cite{fontana2024high}. The algebraic representation revealed to be effective to find the geometrical generators and  characterize  their dependence structure  also in low dimension. In fact, it plays a key role to find our main result in dimension three. We first introduce the geometrical representation, then the algebraic representation and their connection. Section \ref{sec:conc} concludes.

\subsection{Geometrical representation}\label{sec.subsection}
Let $\FFF_3$ be the simplex of probability mass functions (pmfs) of $3$-dimensional Bernoulli random variables. Let us consider the Fr\'echet class  $\FFF_3(p)\subseteq\FFF_3$ of pmfs  in $\FFF_3$ which have the same Bernoulli marginal distribution of parameter $p$, $B(p)$.
	We assume throughout the paper that $p$ is rational, i.e.  $p\in \mathbb{Q}$. Since $\mathbb{Q}$ is dense in $\RR$, this is not a limitation in applications. We also make the non-restrictive (for symmetry) assumption that $p\leq 1/2$.  
	Let  $\XX=(X_1, X_2, X_3)$ be   a $3$-dimensional Bernoulli  random vector, the notation   $\XX\in \FFF_3(p)$  indicates that $\XX$  has pmf $f\in \FFF_3(p)$.
	The column vector which contains the values of  $f$ over $\design=\{0, 1\}^3$ is   $\ff = (f_1,\ldots, f_{8})=(f_{\xx}:\xx\in\design):=(f(\xx):\xx\in\design)$. We make the non-restrictive hypothesis that the set $\design$ of $2^3$ binary vectors is ordered according to the reverse-lexicographical criterion. We denote with $\designk=\{\xx\in \design: \sum_{j=1}^3x_j=k\}$, $k=0,1,2,3$. It follows that $\{\design_0, \design_1, \design_2, \design_3\}$ is a partition of $\design$.

Let $\XX\in \FFF_3(p)$. 
 As described in \cite{fontana2018representation}, 
given    $\XX\in\FFF_3(p)$, the three conditions $E[X_i]=p$, $i=1,2,3$ translate in the following linear system in $\ff$:
\begin{equation}\label{L-sist}
\sum_{\xx\in \design: x_k=1}(1-p)f(\xx)-\sum_{\xx\in \design: x_k=0}pf(\xx)=0,\,\,\, k=1,2,3.
\end{equation}
Let $H$ be the $3 \times 8$ matrix of the coefficient of the system \eqref{L-sist}.
The class  $\FFF_3(p)$  is the following  convex polytope
\begin{equation}  \label{cone}
\FFF_3(p)=\{\ff\in \RR^{8}: H\ff=0, f_i\geq 0, \sum_{i=1}^{8}f_i=1\}.
\end{equation}
Each point $\ff\in \FFF_3(p)$ is a convex combination of a set of
generators referred to as the extremal points of the convex polytope $\FFF_3(p)$. Formally, for any $\ff\in \FFF_3(p)$
there exist  $n_p$ extremal points (also called extremal pmfs) $\rr_i\in \FFF_3(p)$ and coefficients $\lambda_1\geq0,\ldots, \lambda_{n_p}\geq0$ summing up to one, such that
$$\ff=\sum_{i=1}^{n_p}\lambda_i\rr_i.$$ We also call extremal points the random vectors $\RRR_i\sim \rr_i.$ Proposition 2.1 in \cite{fontana2018representation} implies that the extremal points of $\FFF_3(p)$ have at most four non-zero values.


\subsection{Algebraic representation}
This section recalls the algebraic representation introduced in \cite{fontana2024high}, adapted to the three dimensional case. It was introduced with the aim of finding the extremal points in high dimension, where the number of extremal points make the computational approach infeasible. Actually, this representation plays a key role also in dimension three to find an analytical expression for all the extremal points as functions of $p$.  

We define a map $\mathcal{H}$ from $\FFF_3(p)$ to the polynomial ring with rational coefficients  $\mathbb{Q}[x_1, x_2]$, as follows 
 \begin{equation}\label{Eh}
\begin{split}
&\mathcal{H}: \FFF_3(p)\rightarrow \mathbb{\QQ}[x_1,x_2]\\
&\mathcal{H}:\ff\rightarrow P_f(\xx)=\mm(\xx) \ff,
\end{split}
\end{equation}
 where $\mm(\xx) =(\mm_+(\xx)||\mm_-(\xx))$ is the row vector obtained concatenating $\mm_+(\xx)$ and  $\mm_-(\xx)$, with
 $\mm_+(\xx):=(1, x_1,x_2, x_1x_2)$ and $\mm_-(\xx)=(-x_1x_2+c,-x_2+c, -x_1+c, -1+c)^T$, where $p=\tfrac{s}{t}$ and $c=\tfrac{2s-t}{s}$.
We denote by $\mathcal{C_H}$ the image of $\FFF_3(p)$ under the map $\mathcal{H}$.
By construction, rearranging the coefficients of $P_f(\xx)$ yields:

\begin{equation}\label{pol}
P_f(\xx)=\mm(\xx)\ff=\sum_{\alpha\in\{0,1\}^{2}}a_{\alpha}x^{\alpha}, 
\end{equation}
where $ a_{\alpha}\in \QQ$ are linear combinations of the elements of $\ff$. 
The map  $\mathcal{H}$ is not injective, as noted in  \cite{fontana2024high} where, the authors 
 present a three-step algorithm to uniquely determine a pmf $\ff_P=(f_1,\ldots, f_{8})\in \FFF_3(p)$ associated to a given polynomial $P(\xx)=\sum_{\alpha\in\{0,1\}^{2}}a_{\alpha}\xx^{\alpha}$. We refer to $\ff_P$ as the type-0 pmf associated to $P(\xx)$. The following proposition characterizes both the counter-image of a polynomial $P(\xx)$ and  the kernel of the map $\mathcal{H}$. We state the proposition for a general dimension $d$.
\begin{proposition}\label{ker}
 Given a polynomial $P(\xx)=\sum_{\alpha\in\{0,1\}^{d}}a_{\alpha}\xx^{\alpha}$, $$\mathcal{H}^{-1}(P(\xx))=\{\lambda\ff_p+(1-\lambda)\ee_k, \,\,\ \ee_k\in Ker (\mathcal{H}), \lambda\in[0,1]\},
$$
where $\ff_p$ is the type-0 pmf associated to $P(\xx)$. A basis of $ Ker (\mathcal{H})$ is
\begin{equation*}
\begin{split}
\mathcal{B}_K&=\{(1-p,0,\ldots, 0,p); (1-2p,p,0,\ldots,0,p,0); (1-2p,0,p,0,\ldots0,,p,0,0);\\
&\ldots (1-2p,0,\ldots, 0,p,p,0\ldots,0)\}.
\end{split}
\end{equation*}
\end{proposition}
An important result about the algebraic representation is that all the polynomials in $\mathcal{C_H}$ are linear combinations with scalar coefficients of a specific class of polynomial, that are referred as the algebraic generators and  that we call fundamental polynomials. The fundamental polynomials of the class $\FFF_3(p)$ are 
\begin{equation*}F^{\pm}(\xx)=\pm\big(x_{1} x_{2}-x_1-x_2+1\big).
\end{equation*}
\begin{proposition}
The polynomials $P_f(\xx)=\mm(\xx)\ff\in \mathcal{C}_{\mathcal{H}}$ are linear combinations with scalar coefficients of the fundamental polynomials. 
\end{proposition}

We conclude this section with a result that establishes a connection between geometrical and algebraic generators of the class $\FFF_3(p)$.
\begin{proposition}\label{type0}
Let $\ff\in \FFF_3(p)$ be the type-0 pmf associated  to a fundamental polynomial $F^{\pm}(x)$. The pmf $\ff$ is an extremal probability mass function.
\end{proposition}
In the symmetric case where $p=0.5$, \cite{mutti2023symmetric} proves that all the extremal points of the class $\FFF_d(0.5)$ - that, given $p$,  in dimension three can be found explicitly with 4ti2- are type-0 or belong to the kernel.

\section{Extremal points} \label{sec:gen3}

Let $\DDD(3p)$ be the class of discrete pmfs on $\{0,1,2,3\}$ with mean $3p$. 
Any pmf $p\in \DDD(3p)$ is the distribution of the sum of  the components of  at least one $3$-dimensional Bernoulli random vector $\XX\in \FFF_3$ (see e.g. \cite{fontana2020model}). In general, behind any discrete pmf there are infinite Bernoulli vectors $\XX\in \FFF_d$ (\cite{fontana2025bernoulli}). Formally, we define the following map between $\FFF_3(p)$ and $\DDD(3p)$.
\begin{equation}\label{funsum0}
\begin{split}
s:&\FFF_3(p)\rightarrow \DDD(3p)\\
   &  \ff\rightarrow s_{\ff},
\end{split}
\end{equation}
where $s_{\ff}:=s(\ff)$ is the distribution of the sum $S:=X_1+X_2+X_3$ with $\XX\sim \ff\in \FFF_3(p)$.
 Then, the intersection of the support of $\ff\in \FFF_3(p)$ with  $\design_k$ is not empty iff $s_{\ff}(k)\neq 0$.

Our main result is the following Theorem \ref{Prop:EP} that analytically finds all the extremal points of $\FFF_3(p)$. 
 The extremal points are reported in Table \ref{tab:exp1} for $p \leq \frac{1}{3}$ and in Table \ref{tab:exp2} for $p \in (\frac{1}{3}, \frac{1}{2}] $.





\begin{theorem}\label{Prop:EP}
The extremal points of $\FFF_3(p)$ are given by the columns of Table \ref{tab:exp1} for $p \leq \frac{1}{3}$, and by the columns of Table \ref{tab:exp2} for $\frac{1}{3} < p \leq \frac{1}{2}$.
\end{theorem}

\begin{proof}
    First, we consider the case $p \leq \frac{1}{3}$ and refer to Table \ref{tab:exp1}. As proved in \cite{fontana2024high}, $\rr_1,\ldots,\rr_6$ are extreme pmfs, since $\rr_1, \rr_2, \rr_3$, and $\rr_5$ form a basis of $Ker(\mathcal{H})$, while $\rr_4$ and $\rr_6$ are the type-0 fundamental pmfs, i.e. type-0 pèmfs associated to the fundamental polynomials $F^+(x)$ and $F^-(x)$ respectively.
   It follows that what remains to be proved is that there are not other extreme pmfs. 
    For simplicity of notation, we write $s$ for $s_{\ff}$ throughout the proof.
    
    The mean constraint $E[S] = 3p $, if $p \leq \frac{1}{3}$, implies that the sum $S$ has mean between $0$ and  $1$. Accordingly, the pmf $s$ may have support on two points ($\{0,1\}$, $\{0,2\}$, or $\{0,3\}$), on three points ($\{0,1,2\}$, $\{0,2,3\}$, or $\{0,1,3\}$), or on the full set $\{0,1,2,3\}$. Let us consider the cases $\{0, i\}$, $i=1,2,3$.
    The condition on the mean, implies that there is a unique pmf $s_{0,i}$, supported on  $\{0, i\}$, with $i \in \{1, 2, 3\}$.
It is defined as $s_{0,i}(0) = 1 - \frac{3p}{i}$, $s_{0,i}(i) = \frac{3p}{i}$, and $s_{0,i}(k)=0$ for $k\in\{1,2,3\}, \, k\neq i$.

Let us now consider the case $i = 1$. We have $f(0,0,0) = s_{0,1}(0) = 1 - 3p$, and $f(1,0,0) + f(0,1,0) + f(0,0,1) = s_{0,1}(1) = 3p$.
Since $E[X_i] = p$ for $i = 1, 2, 3$, the only admissible pmf satisfies $f(1,0,0) = f(0,1,0) = f(0,0,1) = p$.
This corresponds to $\rr_6$ in Table \ref{tab:exp1}. 
     The case $i = 2$ is analogous, yielding the extremal point $\rr_4$ in Table \ref{tab:exp1}.
If $i = 3$, we have $f(0,0,0) = s_{0,3}(0) = 1 - p$ and $f(1,1,1) = s_{0,3}(3) = p$. This corresponds to $\rr_5$ in Table \ref{tab:exp1}.
    
If $s$ has support on $\{0,1,2\}$, then $\ff$ has support on $\design_0 \cup \design_1 \cup \design_2$. Since $\ff$ must be an extremal point, by Proposition 2.1 in \cite{fontana2024high}, it must have support on at most four points: one point is $(0,0,0)$, at least one point lies in $\design_1$, and at least one point lies in $\design_2$. Therefore, it has support on at least three points. Assume that the support consists of three points. Let $\xx \in \design_1$ be such that $f(\xx) \neq 0$.
The mean constraints $E[X_1] = E[X_2] = E[X_3] = p$ imply that $f(\xx) = p$ and $f(\boldsymbol{1} - \xx) = p$, where $\boldsymbol{1} - \xx=(1-x_1, 1-x_2, 1-x_3)$.
By choosing $\xx = (1,0,0)$, $(0,1,0)$, and $(0,0,1)$, we obtain the extremal points $\rr_1$, $\rr_2$, and $\rr_3$, respectively.
We now consider the case in which the support consists of four points.
Since the means are equal, this configuration is incompatible with having two points in $\design_1$ or two points in $\design_2$.
Using similar arguments, we can show that for the supports $\{0,1,3\}$ and $\{0,2,3\}$, since $f(0,0,0) \neq 0$ and $f(1,1,1) \neq 0$, the mean constraints imply that the support of $f$ must contain at least three additional points; hence, it cannot be an extremal point.

If $\ss$ has full support, then $\ff$ has support on $\design_0 \cup \design_1 \cup \design_2 \cup \design_3$, that is, on at least four points.
Since $\ff$ must be an extremal point, it must have support on exactly four points: $(0,0,0)$, one point in $\design_1$, one point in $\design_2$, and $(1,1,1)$.
The mean constraints imply the following structure:
\begin{equation}
\begin{split}
    f(0,0,0)&=1-2p+\theta\\
    f(\xx)&=p-\theta,\\
    f(1-\xx)&=p-\theta,\\
    f(1,1,1)&=\theta.
    \end{split}
\end{equation}
where $0 < \theta < p$, $\xx \in \design_1$ and then $1-\xx \in \design_2$.
By Lemma 2.3 in \cite{terzer2009large}, we can prove that it is not an extremal point.

We now consider the case $\frac{1}{3} < p \leq \frac{1}{2}$ and refer to Table \ref{tab:exp2}.
As proved in \cite{fontana2024high}, $\rr_1, \ldots, \rr_5$, and $\rr_9$ are extreme pmfs, since $\rr_1, \rr_2, \rr_3$, and $\rr_5$ form a basis of $\text{Ker}(\mathcal{H})$, while $\rr_4$ and $\rr_9$ are the type-0 fundamental pmfs. In \cite{IES2025} it is proved that also $\rr_6, \rr_7$ and $\rr_8$ are extremal points.
 Also in this case, what remains to be proved is that there are not other extreme pmfs.
 
The condition on the mean of the sum $S$, when  $\tfrac{1}{3} < p \leq \tfrac{1}{2}$ is $1 < E[S] = 3p < \tfrac{3}{2}$. It implies that $s$, can have support on
two points ($\{0,2\}$, $\{0,3\}$, $\{1,2\}$, or $\{1,3\}$),
on three points ($\{0,1,2\}$, $\{0,2,3\}$, $\{0,1,3\}$, or $\{1,2,3\}$),
or full support $\{0,1,2,3\}$.

The cases $\{0,2\}$ and $\{0,3\}$ are the same as those considered for $p < \tfrac{1}{3}$, leading to $\rr_4$ and $\rr_5$ in Table \ref{tab:exp2}.
The cases $\{0,1,2\}$, $\{0,2,3\}$, $\{0,1,3\}$, and $\{0,1,2,3\}$ have already been discussed for $p < \tfrac{1}{3}$.
We find that $\{0,1,2\}$ leads to $\rr_1$, $\rr_2$, and $\rr_3$, whereas the other cases are incompatible with the mean constraints $E[X_1] = E[X_2] = E[X_3] = p$.
We are thus left with the cases $\{1,2\}$ and $\{1,3\}$.
For $\{1,3\}$, the condition on the mean implies that if $f(1,1,1) \neq 0$, then $f(\xx) \neq 0$ for every $\xx \in \design_1$, yielding the extremal point $\rr_9$.
We now consider the case $\{1,2\}$.
We only need to prove that $\rr_6$, $\rr_7$, and $\rr_8$ are the only extremal points of $\FFF_3(p)$ with support $\design_1 \cup \design_2$.
It follows that $\ff$ has support on at least two points, at least one $\xx \in \design_1$ and one $\xx \in \design_2$.

We consider the following possible cases:
\begin{enumerate}
   
    \item Two points: the condition on the means implies that $\ff$ satisfies $\ff(\xx)=\ff(\boldsymbol{1}-\xx)=p,$ with $\xx\in \design_1$. Since the components of $\ff$ sums up to one, this is not possible with $p\neq1/2$.
    \item Three points: we have two cases. Case 1, one point in $\design_1$ and two points in $\design_2$, but this case is not compatible with equal means.  Case 2, two points in $\design_1$ and one point in $\design_2$, but also this case is not compatible with equal means. 
    \item Four points: the only case compatible with the condition on the means is three point on $\design_1$ and one point on $\design_2$ and we find the three pmfs $\rr_6$, $\rr_7$ and $\rr_8$.
\end{enumerate}

Therefore, all the extremal points are those listed in Tables \ref{tab:exp1} and \ref{tab:exp2}.
\end{proof}
    \begin{remark}
        Notice that for $p=1/2$ the extremal points $\rr_6,$ $\rr_7$ and $\rr_8$ coincides with the extremal points $\rr_1,$ $\rr_2$ and $\rr_3$ and belong to the kernel of $\mathcal{H}.$
    \end{remark}
{\footnotesize{
\begin{table}[h]
	\centering
		\begin{tabular}{ccc|rrrrrrrrr}
$\xx_1$ &	$\xx_2$ &	$\xx_3$ &	$\rr_1$ & 		$\rr_2$ & 		$\rr_3$ & 		$\rr_4$ & 		$\rr_5$ & 		$\rr_6$\\
		\hline
0 &	0 &	0 &	$1-2p$ &	$1-2p$ &	$1-2p$&	$1-3\frac{p}{2}$ &	$1-p$&	${1-3p}$\\
1 &	0 &	0 &	0 &	0 &	$p$ &	0 &	0 &	$p$\\
0 &	1 &	0 &	0 &	$p$ &	0 &	0 &$0$&$p$\\
1 &	1 &	0 &$p$ &	0 &	0 &	$\frac{p}{2}$&	0 &0\\
0 &	0 &	1 &	$p$&	0 &	0 &	0&	0 &$p$\\
1 &	0 &	1 &	0 &$	p$ &0 &	$\frac{p}{2}$ &	0 &	0\\
0 &	1 &	1 &	0 &	0 &$	p $&	$\frac{p}{2}$ &	0 &0\\
1 &	1 &	1 &	0 &	0 &	0 &	0 &	$p$&0\\
		\end{tabular}
	\caption{Extremal pmfs  of $\FFF_3(p)$ for $p\leq1/3.$}
	\label{tab:exp1}
\end{table}}

\footnotesize{
\begin{table}[h]
	\centering
		\begin{tabular}{ccc|rrrrrrrrr}
$\xx_1$ &	$\xx_2$ &	$\xx_3$ &	$\rr_1$ & 		$\rr_2$ & 		$\rr_3$ & 		$\rr_4$ & 		$\rr_5$ & 		$\rr_6$ &$\rr_7$&$\rr_8$& $\rr_9$\\
		\hline
0 &	0 &	0 &	$1-2p$ &	$1-2p$ &	$1-2p$&	$1-3\frac{p}{2}$ &	$1-p$&	0 &0&0&0\\
1 &	0 &	0 &	0 &	0 &	$p$ &	0 &	0 &	$1-2p$&$1-2p$&$p$&$\frac{1-p}{2}$\\
0 &	1 &	0 &	0 &	$p$ &	0 &	0 &	0 &	$1-2p$&$p$&$1-2p$&$\frac{1-p}{2}$\\
1 &	1 &	0 &$p$ &	0 &	0 &	$\frac{p}{2}$&	0 &	$3p-1$ &0&0&0\\
0 &	0 &	1 &	$p$&	0 &	0 &	0&	0 &	$p$&$1-2p$ &$1-2p$&$\frac{1-p}{2}$\\
1 &	0 &	1 &	0 &$	p$ &0 &	$\frac{p}{2}$ &	0 &	0 &$3p-1$&0&0\\
0 &	1 &	1 &	0 &	0 &$	p $&	$\frac{p}{2}$ &	0 &	0&0&$3p-1$&0\\
1 &	1 &	1 &	0 &	0 &	0 &	0 &	$p$ &	0&0 &0&$\frac{3p-1}{2}$\\
		\end{tabular}
	\caption{Extremal pmfs  of $\FFF_3(p)$ for $p\in(1/3, 1/2].$}
	\label{tab:exp2}
\end{table}}}
The extremal points in $\FFF_3(1/2)$ are listed in Table 1 in \cite{mutti2023symmetric} and discussed in that paper. For the reader's convenience, they are $\rr_i,\,\,\, i=1,2,3,4,5,9$ in Table \ref{tab:exp2}.  
\begin{remark}
    As proved in \cite{fontana2024high}, $\rr_1, \rr_2, \rr_3$, and $\rr_5$ of both Table \ref{tab:exp1} and \ref{tab:exp2} form a basis of $\text{Ker}(\mathcal{H})$, and  $\rr_4$ and $\rr_6$ of Table \ref{tab:exp1} and $\rr_4$ and $\rr_9$ of Table \ref{tab:exp2} are the type-0 fundamental pmfs.
    \end{remark}




\section{Extremal correlations}\label{ExtrCorr}
The weakest notion of dependence for a random vector $\XX$ are the pairwise correlations $\rho(X_i, X_j)$, that if $\XX\in \FFF_d(p)$ reads
\begin{equation}
   \rho(X_i, X_j)=\frac{\mu_{ij}-p^2}{p(1-p)},
    \end{equation}
    where $\mu_{ij}=E[X_iX_j]$, $i,j=1,\ldots, d, \, i\neq j$.
 We say that a random vector  $\XX$  
is pairwise positive  correlated (P-PC) if
    $\rho(X_i,X_j)\geq 0,\,\,\, i\neq j,$  and it is pairwise negative 
    correlated (P-NC) if
    $\rho(X_i,X_j)\leq 0,\,\,\, i\neq j.$
Since
\begin{equation}
\mu_{ij}=\frac{\partial^2 g(\boldsymbol{1})}{\partial x_i\partial x_j},
\end{equation}
where $g(\xx):=E[\xx^{\XX}]$ is the  probability generating function (pgf) of $\XX$. 
We can therefore  find the pairwise correlations of $\XX$ by using its pgf.

From \cite{fontana2018representation} we know that correlations reach their bounds on the extremal points of $\FFF_3(p)$. Furthermore, if $\ff\in \FFF_3(p)$ and $\ff=\sum_{i=1}^n\lambda_i\rr_i$, then $\rho(X_h, X_k)=\sum_{i=1}^n\lambda_i\rho(R_{ih},R_{ik})$ with $\RRR_i\sim \rr_i$, where $n$ is the proper number of extremal pmfs.
Therefore to characterize the admissible pairwise correlations in $\FFF_3(p)$ we have to characterize the pairwise correlations of the extremal points.

\subsection{Pairwise correlation of the extremal points.}

Since Proposition \ref{ker} provides the general expression of the kernel generators, we can  find their pairwise correlations in any dimension and for any $p\leq1/2$.
\begin{proposition}\label{KerProp}
    If $\rr\in Ker \mathcal{H}$ then $\rho_{ij}=1$ or $\rho_{ij}=-\frac{p}{1-p}.$
\end{proposition}
\begin{proof}
We refer to $\mathcal{B}_K$, a basis of $Ker(\mathcal{H})$ as defined in Proposition \ref{KerProp}. It is evident that the pairwise correlation  of the extremal pmf of $Ker(\mathcal{H})$ defined as $(1-p,0,\ldots,0,p)$ is one for any pair of variables.
Now, given $\xx=(1,\ldots, 1, 0,\ldots, 0)$, we consider  $\rr_{\xx},$ that is the extreme pmf of $Ker(\mathcal{H})$ defined as $\rr_{\xx}(\boldsymbol{0})=1-2p$, $\rr_{\xx}(\xx)=\rr_{\xx}(\boldsymbol{1}-\xx)=p$, and zero otherwise. 
The pgf of $\rr_{\xx}$ is
\begin{equation}
g(\xx)=1-2p+px_1\cdots x_k+px_{k+1}\cdots x_d
\end{equation}
If $i,j\leq k$ or $j, j\geq k+1$ we have
\begin{equation}\label{Ray}
\mu_{ij}=\frac{\partial^2 g(\boldsymbol{1})}{\partial x_i\partial x_j}=p,
\end{equation}  
and therefore we have  
$\rho_{ij}=1,$
if $i\leq k$ and $j\geq k+1$ and vice versa, we have 
\begin{equation}\label{Ray}
\mu_{ij}=\frac{\partial^2 g(\boldsymbol{1})}{\partial x_i\partial x_j}=0,
\end{equation} 
and therefore we have  
$\rho_{ij}=-\frac{p}{1-p}.$ All the other cases corresponding to the remaining elements of  $\mathcal{B}_k$, provide the same result. 
\end{proof}
We now go back to dimension three and we consider  the type-0 fundamental extremal points. If $p<1/3$, they are  $\rr_4$ and $\rr_6$ in Table \ref{tab:exp1}. The pgf of $\rr_4$ is 
\begin{equation}
g_4(\xx)=1-3\frac{p}{2}+\frac{p}{2}(x_1x_2+x_1x_3+x_2x_3),
\end{equation}
therefore 
\begin{equation}
\mu_{ij}=\frac{\partial^2 g_4(\boldsymbol{1})}{\partial x_i\partial x_j}=\frac{p}{2}, \, \, i,j=1,2,3,\,\,\,i\neq j, 
 \end{equation}
 and 
 \begin{equation}
\rho_{ij}=\frac{1-2p}{2(1-p)}, \, \, i,j=1,2,3,\,\,\,i\neq j. 
 \end{equation}
 Notice that if $p\leq1/2$ (that is our assumption) $\rho_{ij}>0$.
 We now consider  $\rr_6$, its pgf is 
\begin{equation}
g_6(\xx)=1-3p+{p}(x_1+x_2+x_3),
\end{equation}
therefore it is  straightforward to find  $\rho_{ij}=-\frac{1}{(1-p)}$ for all $i, j$.
 We notice that  the type-0 fundamental extremal points, $\rr_4$ and $\rr_6$, are P-PC and  P-NC, respectively.
 If $p>1/3$   the type-0 fundamental are  $\rr_4$ and $\rr_9$ in Table \ref{tab:exp2}. While, if $p\in(1/3, 1/2]$, $\rr_4$ has the same expression and therefore the same correlations of $\rr_4$ with $p<1/3$, to find pairwise correlations for $\rr_9$ we compute its pgf
\begin{equation}
g_9(\xx)=\frac{1-p}{2}(x_1+x_2+x_3)+\frac{3p-1}{2}x_1x_2x_3.
\end{equation}
  Therefore 
\begin{equation}
\mu_{ij}=\frac{\partial g^2_9(\boldsymbol{1})}{\partial x_i\partial x_j}=\frac{3p-1}{2}, \, \, i,j=1,2,3,\,\,\,i\neq j, 
 \end{equation}
 and therefore
 \begin{equation}
\rho_{ij}=\frac{3p-1-2p^2}{2p(1-p)}, \, \, i,j=1,2,3,\,\,\,i\neq j. 
 \end{equation}
 We observe that for $p\leq1/2$ (that is our assumption) we have $\rho_{ij}<0$. Therefore,  we have proved the following proposition.
 \begin{proposition}\label{Prop:Typ0corr}
     The type-0 fundamental extremal pmf asscuated to $F^+(x)$ ($F^-(x)$)  is pairwise positively (negatively) correlated.
 \end{proposition}
The only missing case are the three extremal pmfs $\rr_6, \rr_7$, and $\rr_8$ with support on $\design_1\cup\design_2$ for $p\in(1/3, 1/2]$.
  The following proposition provides their pairwise correlations.
  \begin{proposition}
    If $\rr$ has support on $\design_1\cup\design_2$  and $p\in(1/3, 1/2]$ then $\rho_{i_1,i_2}=\frac{3p-1-p^2}{p(1-p)}$,  $\rho_{i_1,i_3}=\rho_{i_2,i_3}=-\frac{1}{(1-p)}$,  where $\{i_1, i_2, i_3\}=\{1,2,3\}.$
  \end{proposition}
  \begin{proof}
      We consider $\rr_6,$ the other cases are obtained by permutation of the indices. 
The pgf of $\rr_6$ is
\begin{equation}
g_6(\xx)=(1-2p)x_1+(1-2p)x_2+(3p-1)x_1x_2+px_3
\end{equation}
We have 
\begin{equation}\label{Ray}
\mu_{12}=\frac{\partial^2 g_6(\boldsymbol{1})}{\partial x_1\partial x_2}=3p-1,
\end{equation}  
\begin{equation}\label{Ray}
\mu_{13}=\frac{\partial^2 g_6(\boldsymbol{1})}{\partial x_1\partial x_3}=0,\,\,\,\,\mu_{23}=\frac{\partial^2 g_6(\boldsymbol{1})}{\partial x_2\partial x_3}=0
\end{equation}  
Therefore $\rho_{13}=\rho_{23}=-\frac{1}{(1-p)}<0$.  It is easy to check that 
$\rho_{1,2}=\frac{3p-1-p^2}{p(1-p)}$ and $\rho_{1,2}>0$.

  \end{proof}
 We  have characterized the P-NC of all the extremal points. It is interesting to notice the type-0 fundamental pmfs are the only extremal points that can be P-NC for any $p$.

\subsection{Extremal positive and negative dependence}

This section studies extremal positive and negative dependence, specifically we characterize the pmfs with the strongest positive or negative dependence.
The strongest positive dependence is  the scenario of perfect positive dependence among components. This is possible in any Fr\'echet class of distributions, and it corresponds to the \textit{upper Fr\'echet bound}, that is the distribution of the comonotonic vector. The upper Fr\'echet bound of $\FFF_3(p)$ is $\rr_5$ in both Tables \ref{tab:exp1} and \ref{tab:exp2}, and it corresponds to perfect pairwise correlation, i.e. $\rho_{ij}=1$ for any $i,j=1,2,3$.

When studying  extremal negative dependence, the starting point is the definition of countermonotonicity.
\begin{definition}
    A bivariate random vector $(X,Y)$ is said to be countermonotonic if
    \begin{equation*}
	\PP \big\{(X_1-X_2)(Y_1-Y_2) \leq 0 \big\} = 1,
    \end{equation*}
    where $(X_1,Y_1)$ and $(X_2,Y_2)$ are two independent copies of $(X,Y)$.
\end{definition}

Although this definition provides a clear characterization of extremal negative dependence for two-dimensional Fr\'echet classes, there is no unique and straightforward generalization to higher dimensions. 
Indeed, in dimension two, the distribution of a countermonotonic vector is the lower Fr\'echet bound. In dimensions higher than two, the lower Fr\'echet bound is not always a distribution function, see \cite{dall2012advances}.
Several approaches have been proposed to define notions of minimal dependence in Fr\'echet classes of dimension higher than 2. 
These notions are known as extremal negative dependence concepts; see \cite{puccetti2015extremal}.
In \cite{puccetti2015extremal}, the authors introduce the notion of $\Sigma$-countermonotonicity.
This definition is significant because it extends countermonotonicity in dimension higher than two in such a way that  every Fr\'echet class admits a $\Sigma$-countermonotonic random vector.
\begin{definition} \label{def:SigmaCountermonotonic}
    A $d$-dimensional random vector $\YY = (Y_1,\ldots,Y_d)$ is said to be $\Sigma$-countermonotonic if, for every subset $J \subseteq \{1,\ldots,d\}$, the pair $(\sum_{j \in J} Y_j,\sum_{j \notin J} Y_j)$ is countermonotonic, with the convention $\sum_{j \in \emptyset} Y_j = 0$.
\end{definition}
If a Fr\'echet class admits a countermonotonic vector,  this is the only $\Sigma$-countermonotonic one and its distribution is the lower Fr\'echet bound.

In  $\FFF_3(p)$ the 
 countermonotonic vector exists if and only if 
 $p<1/3$ (see, e.g.,  \cite{fontana2018representation}). It is $\rr_6$ of Table \ref{tab:exp1} and its probability density function (pdf) is the lower Fr\'echet bound. 
Extremal negative dependence is more challenging to characterize if $p>1/3$, since the class does not admit a countermonotonic vector and the lower Fr\'echet bound is no longer a pdf.

We now characterize the $\Sigma$-countermonotonic pmfs in $\FFF_3(p)$. We need to introduce another definition of extremal negative dependence, that is necessary to prove our main result.  
We need to introduce the convex order first.
\begin{definition}
    Given two random variables $Y_1$ and $Y_2$ with finite means, $Y_1$ is said to be smaller than $Y_2$ under the convex order (denoted $Y_1 \le_{cx} Y_2$) if $E \{\phi(Y_1)\} \leq E\{\phi(Y_2)\}$, for all real-valued convex functions $\phi$ for which the expectations are finite.
\end{definition}
We now define a class of vectors whose sum is minimal in the convex order, and we call them $\Sigma_{cx}$-smallest elements.
\begin{definition} \label{def:SigmaCX}
    A $\Sigma_{cx}$-smallest element in a class of distributions $\FFF$ is a random vector $\YY = (Y_1,\dots,Y_d)$ with distribution in $\FFF$ such that
    \begin{equation*}
        \sum_{j=1}^d Y_j \leq_{cx} \sum_{j=1}^d Y'_j,
    \end{equation*}
    for any random vector $\YY'$ with distribution in $\FFF$.
\end{definition}
The $\Sigma_{cx}$-smallest elements are the vectors $\YY$  that minimize aggregate risk, meaning that the sum 
 $S:=Y_1+\cdots+Y_d$ is minimal in the convex order within a given class of distributions.
The convex order is a variability order; thus, a random variable that is minimal in the convex order is a minimal risk random variable. 
Therefore, the purpose of this extremal negative dependence is to minimize the aggregate risk. 
In general,  we can not find a $\Sigma_{cx}$-smallest element in any Fr\'echet class, but \cite{cossette2024generalized} proves that Bernoulli Fr\'echet classes always admit a $\Sigma_{cx}$-smallest element. Furthermore, \cite{cossette2025extremal} proves that in any Bernoulli Fr\'echet class a pmf is $\Sigma$-countermonotonic iff it is $\Sigma_{cx}$-smallest. This result plays a key role in characterizing extremal negative dependence in $\FFF_3(p)$, $p\in(1/3, 1/2]$. In fact, the proof of the following Proposition \ref{mincx} relies on the characterization of the $\Sigma_{cx}$-smallest elements $\XX$  through the distribution of the sum of their components $S:=\sum_{i=1}^3X_i$. Using \cite{fontana2018representation} we can easily prove that if $\XX\in \FFF_3(p)$ is a $\Sigma_{cx}$-smallest element then $S$ has pmf $s_{1, 2}$ given by:
\begin{equation}\label{mincx}
s_{1, 2}(y)=\begin{cases}{2-3p} \,\,\, y=1\\
   {3p-1}  \,\,\, y=2.
\end{cases}
\end{equation}
The following Proposition \ref{mincx} characterizes the class of $\Sigma$-countermonotonic pmfs in $\FFF_3(p)$.

\begin{proposition}\label{Sigmacx}
    The $\Sigma$-countermonotonic pmfs in $\FFF_3(p)$ are a convex polytope whose generators are all the $\Sigma_{cx}$-smallest extremal points of $\FFF_3(p)$, they are:
    \begin{enumerate}
        \item $\rr_6$ in Table \ref{tab:exp1} if $p\in(1/3, 1/2]$,
        \item $\rr_6$, $\rr_7$ and $\rr_8$ in Table \ref{tab:exp2} if $p>1/3$.
    \end{enumerate}
\end{proposition}
\begin{proof}
    The case $p<1/3$ is straightforward since $\rr_6$ is the Lower Fr\'echet bound and therefore the only $\Sigma$-countermonotonic pmf in the class, and it is an extremal point.
    Assume that $p\in(1/3, 1/2]$. Theorem 4.1  in \cite{cossette2025extremal} prove that a Bernoulli pmf $\ff$ is $\Sigma$-countermonotonic if and only if it is $\Sigma_{cx}$-smallest in its Fr\'echet class and  Proposition 4.2 in \cite{cossette2025extremal} proves that the $\Sigma_{cx}$-smallest elements are a convex polytope whose extremal points are the extremal points of $\FFF_3(p)$ that are $\Sigma_{cx}$-smallest. Therefore we only have to prove that $\rr_6$, $\rr_7$ and $\rr_8$ are the only extremal points of $\FFF_3(p)$ that are $\Sigma_{cx}$-smallest elements of $\FFF_3(p)$.
   If $\ff$ is $\Sigma_{cx}$-smallest and $\XX\sim \ff$, then $S=\sum_{i=1}^3X_i$ has support on 1 and 2 (see, e.g.  Proposition 5.1 in \cite{fontana2024high}). Therefore $\ff$ has support on $\design_1\cup \design_2$.  Theorem \ref{Prop:EP} proves that the only extremal points with support on $\design_1\cup \design_2$ are $\rr_6$, $\rr_7$ and $\rr_8$.

\end{proof}


    

We call $\FFF_3^{\Sigma}(p)$ the polytope of $\Sigma$-countermonotonic pmfs in $\FFF_3(p)$.
We notice that extremal negative dependence does not imply P-NC.
However, all the random vectors in $\FFF_3^{\Sigma}(p)$ minimize any convex measure of aggregate risk, such as the variance of the sum $S$, that, indeed, assumes the same value across $\FFF_3^{\Sigma}(p)$.

As discussed in \cite{cossette2025extremal} for any Fr\'echet class of Bernoulli pmfs,
$\Sigma$-countermonotonic Bernoulli random vectors can have positively correlated pairs of random variables. In fact, in \cite{cossette2025extremal} the authors prove that the mean of the pairwise correlations of any $\ff\in \FFF^{\Sigma}_3(p),$ $p\in(1/3,1/2]$, is constant 
since
\begin{equation} \label{eq:summ}
    \mu_2^+=\sum_{1\leq j_1 < j_2 \leq 3} E(X_{j_1}X_{j_2})
         = 3p - 1.
\end{equation}
Notice that if $p<1/3$, $\FFF^{\Sigma}(p)=\{\rr_6\}$ and we have $\mu^+_2=0$, in fact the lower Fr\'echet bound has  second order moments equal to zero.
Since $\mu_2^+$ is constant, if $p\in (1/3, 1/2]$ and  $\XX$ has some countermonotonic pairs, it will have other pairs with higher --- possibly positive --- correlation, to keep the mean correlation equal to the constant value of the class.
Nevertheless,  there is at least one P-NC pmf in $\FFF_3^{\Sigma}(p)$, the exchangeable pmf with distribution 
\begin{equation}\label{eq:exc}
    \ff^e=\frac{1}{3}\rr_6+\frac{1}{3}\rr_7+\frac{1}{3}\rr_8.
\end{equation}
Exchangeable random vectors have equi-correlated pairs. The equi-correlation of $\ff^e$ in \eqref{eq:exc} is exactly the mean correlation $\rho$ on $\FFF^{\Sigma}_3(p)$, that is
\begin{equation}
    \rho=\frac{-3p^2+6p-1}{3p(1-p)}.
\end{equation}
The next section shows an application example of the effect of the different dependence structures across $\FFF_3^{\Sigma}(p)$.

 
\section{Variance allocation in $\FFF^{\Sigma}_3(p).$}\label{games}
Let $\XX\in \FFF_3^{\Sigma}(p)$, then $\XX$ is both $\Sigma$-countermonotonic and $\Sigma_{cx}$-smallest. Therefore $\XX$ has sums $S=X_1+X_2+X_3$ minimal in convex order, that means that $\XX$ minimize aggregate risk. In particular
the variance of $V(S)$ of $S$  is minimal in $\mathcal{D}(3p)$. However, the dependence structure  of $\XX$ changes across $\FFF^{\Sigma}_3(p)$, as discussed in the previous section. To study the effect of negative dependence on $V(S)$ we consider the marginal contributions of the components $X_i$, $i=1,2,3$ to $V(S)$.
We follow \cite{colini2018variance}, where the authors propose an allocation criterion for the variance of the sum of $n$ possibly dependent variables using a game theory approach. 

We need to introduce some preliminary notions on finite cooperative games. A standard reference for cooperative games is \cite{maschler2020game}. Here we recall the notions and properties useful to our discussion and limited to our framework, specifically we work in dimension three, and therefore we consider three persons games. 
Given a set of players $N=\{1,2,3\}$ we identify the set $\design$ with set of parts of $N$. A cooperative game is a pair $\langle N, \nu\rangle$ where $\nu:\design \rightarrow \RRR$ is such that $\nu(\emptyset)=0.$ Any subset $J\subset N$ is called a coalition and $N$ is the grand coalition. Since the set of players is fixed we call game the function $\nu$. We call $\mathcal{G}(N)$ the class of games with set of players $N$.
Given a random vector $\XX\in \FFF_3(p)$ we consider the associated variance game, as defined in \cite{colini2018variance} in the more general framework of  $n$ players.

Let $J\subseteq N:=\{1,2,3\}$ and $S_J=\sum_{j\in J}X_j$ and let $\nu$ be the cooperative game defined by 
\begin{equation}\label{Vgame}
    \nu(J)=V(S_J), \,\,\, J\subseteq \{1,2,3\}
\end{equation}
where $V(S)$ indicates the variance of $S$.
Each $\XX\sim\ff\in \FFF_3(p)$ defines a game, that we denote with $\nu^{\XX}$ or $\nu^{\ff}.$ We denote with $\nu^j$ the games defined by the extremal points $\rr_j$ in Tables \ref{tab:exp1} and \ref{tab:exp2}.

The value of the grand coalition $N=\{1,2,3\}$ is the variance of the sum $S$  and therefore it is constant and minimal over $\FFF_{3}^{\Sigma}(p)$. 
The Shapley value  is a function $\phi:\mathcal{G}(N)\rightarrow \RR^3$ 
that satisfies the following  properties:
\begin{enumerate}
\item Efficiency-- $\sum_{i=1}^3\phi_i(\nu)=\nu(N)$;
\item symmetry -- if $i$ and $j$ are symmetric, then $\phi_i(\nu)=\phi_j(\nu)$.
\item dummy player -- if player $i$ is a dummy, then $\phi_i(\nu)=0;$
\item linearity -- for $\nu$ and $\mu$ games on $N$ and $\alpha, \beta\in \RRR$ we have
\[\phi(\alpha\nu+\beta\mu)=\alpha\phi(\nu)+\beta\phi(\mu).\]
\end{enumerate}
\cite{shapley1953value} proves that the value exists, it is unique and has the following form:
\begin{equation}
    \phi_i(\nu)=\sum_{J\subset N\setminus\{i\}}\frac{|J|!(N-|J|-1)!}{N!}(\nu(J\cup\{i\})-\nu(J)).
\end{equation}
Looking at its properties, it becomes clear that  the Shapley value represents a fair allocation of what achieved by the gran coalition $N$. In our case, finding the Shapley value of the variance game means finding the marginal contribution of the i-th component of the Bernoulli vector $\XX$ to the variance of the sum. In
\cite{colini2018variance}, the authors  prove that the Shapley value $\phi_i(\nu^{\XX})$ of the variance game $\nu^{\XX}$ is given by
\begin{equation}\phi_i(\nu^{\XX})=\Cov(X_i, S_N).\end{equation}
We first observe  that if $p=1/3$ than by Proposition 3.1 in \cite{colini2018variance} the Shapley value is zero, in fact $S$ is a joint mix ($P(S=pd)=1$, see \cite{wang2011complete}) and $\VaR(S)=0$. 

If $p<1/3$, we only have one $\Sigma$-countermonotonic game, the game $\nu^6$ associated to $\rr_6$ in Table \ref{tab:exp1}. We have
\begin{equation}
    \nu^6(N)=3p-9p^2,
\end{equation}
and
\begin{equation}
    \phi_i(\nu^6)=p-3p^2, \, i=1,2,3.
\end{equation}
We now consider the most  interesting case, that is  $p\in(1/3, 1/2]$. Using \ref{Vgame},
    for any $\XX\sim\ff\in \FFF_3^{\Sigma}(p)$ it holds
    \begin{equation}\label{varmincx}
    \nu^{\XX}(N)=V(S)=9p-9p^2-2.
    \end{equation}
    In fact if $\XX\sim\ff\in  \FFF_3^{\Sigma}(p),$ $S=\sum_{i=1}^3X_i\sim s_{1,2}$ in \eqref{mincx}, and  variance of $S$ is given by \eqref{varmincx}.
The following Proposition \ref{prop:margC} proves that the marginal contribution of the $i$-th variable depends only on the second order moment of the other two variables.
\begin{proposition}\label{prop:margC}
   Let $\XX\sim\ff\in \FFF^{\Sigma}_3(p),$ $p\in(1/3,1/2]$ then
   \begin{equation}
\phi_i(\nu^{\XX})=4p-1-3p^2-\mu_{kl}, 
\end{equation}
where $i,k,l$ are distinct indices in $\{1,2,3\}$.
\end{proposition}
\begin{proof}
We have 
\begin{equation}
\begin{split}
\phi_i(\nu^{\XX})&=\Cov(X_i, S_N)=\sum_{j=1}^3\Cov(X_i,X_j)=\sum_{j=1}^3(E(X_iX_j)-E(X_i)E(X_j))\\
&=\sum_{j=1}^3\mu_{ij}-3p^2=\mu_{ii}+\mu_{2}^+-\mu_{kl}-3p^2=4p-1-3p^2-\mu_{kl}, 
\end{split}
\end{equation}
where $\mu_{ii}=E[X_i^2]=E[X_i]=p$,  $\mu_2^+$ is in \eqref{eq:summ} and $i,k,l$ are distinct indices. 
\end{proof}
The marginal contributions of a vector $\XX\in \FFF_3^{\Sigma}(p)$ can be expressed as a linear convex combination of the marginal contributions of the extremal points $\RRR_i=(R_{i1}, R_{i2}, R_{i3})$,  with pmf $\rr_i$ $i=6,7,8$ in Table \ref{tab:exp2} as proved in the following Proposition \ref{prop:convexmarg}.
\begin{proposition}\label{prop:convexmarg}
    Let $\XX\sim\ff=\sum_{j=6}^8\lambda_j\rr_j\in \FFF^{\Sigma}_3(p)$, then 
    \begin{equation}
\begin{split}
\phi_i(\nu^{\XX})&=\sum_{j=6}^8\lambda_j\phi_i(\nu^j), 
\end{split}
\end{equation}
\end{proposition}
\begin{proof}
If $\ff\in \FFF^{\Sigma}_d(p)$, then $\mu_{jk}=\sum_{j=6}^8\lambda_j\mu^j_{kl}$, as proved in \cite{fontana2018representation}. Therefore, since $\sum_{j=1}^d\lambda_j=1$, we have
    \begin{equation}
\begin{split}
\phi_i(\nu^{\XX})&=4p-1-3p^2-\mu_{kl}=\sum_{j=6}^8{\lambda_j}(4p-1-3p^2-\mu^j_{kl}). 
\end{split}
\end{equation}
\end{proof}

We now compute the marginal contributions to $\nu(N)$ of the components of the extremal points $\RRR_6$, $\RRR_7$ and $\RRR_8.$ 
The marginal contributions corresponding to $\rr_6$ are 
\begin{equation}
    \begin{split}
     \phi_{j}(\nu^6)&=4p-3p^2-1, j=1,2,\\
        \phi_{3}(\nu^6)&=p-3p^2.
    \end{split}
\end{equation}
The contributions $\nu^6_j$, $j=1,2$ are the minimal marginal contributions, since $\mu_{13}=\mu_{23}=0$ are the minimal admissible second order moments of the class, while 
$\nu^6_3$ corresponds to the maximal contribution to the variance of the sum, since $\mu_{12}=3p-1$ is the maximal admissible value for the second order moment in $\FFF_3(p)$. The marginal contributions for $\RRR_7$ and $\RRR_8$ can be obtained simply by permutation of the indices.
The following example illustrates Proposition \ref{prop:convexmarg}, by considering the exchangeable pmf in $\FFF_3(p).$
\begin{example}
    Let us consider the exchangeable pmf $\ff^e\in \FFF^{\Sigma}(p).$ The marginal contributions of the components of the exchangeable vector $\XX^e\sim\ff^e$ are equal, $\phi_i(\nu^e)=3p-3p^2-\frac{2}{3}$. From \eqref{eq:exc} we have
\begin{equation}
    \phi_i(\nu^e)=\frac{1}{3}\phi_i(\nu^6)+\frac{1}{3}\phi_i(\nu^7)+\frac{1}{3}\phi_i(\nu^6),
\end{equation}
and this is easy to check.
\end{example}

We conclude this section by observing that the games associated to different $\Sigma$-countermonotonic vectors have different structures, although they all assume the same value on the grand coalition.
Indeed, the following Proposition \ref{prp:gameprop} is a direct consequence of Proposition 3.2 in \cite{colini2018variance}.
\begin{proposition}\label{prp:gameprop}
The games $\nu^6, \nu^7$ and $\nu^8$ are neither supermodular (i.e. $\nu(I\cup J)+\nu(I\cap J)\geq \nu(I)+\nu(J)$, for all $I,J\subset N$ ) nor submodular (i.e. $\nu(I\cup J)+\nu(I\cap J)\leq \nu(I)+\nu(J)$, for all $I,J\subset N$ ). The game $\nu^e$ is submodular.
\end{proposition}

\begin{remark}
    As proved in  \cite{colini2018variance}, the variance game enjoys the Shapley fusion property, that is:
    $\phi_J(\nu^J)=\sum_{j\in J}\phi_j(\nu)$.
\end{remark}


\section{Conclusion and further research}\label{sec:conc}

We analytically provide the geometrical representation of three dimensional Bernoulli variables. This  opens the way to a complete analysis of their statistical properties. In this paper, we considered the dependence properties of the extremal points, focusing on negative dependence and measuring the impact of different dependence structures on aggregate risk, measured using the variance of the sum.

The importance of a full characterization in dimension three also relies in providing a case where examples and counterexamples can be made to help in the proof or refutation of more general conjectures.

The full representation of Bernoulli variables in dimensions higher than three is not trivial, since the extremal points depends on $p$ in a way much more complicated with respect to the $d=3$ case. For example, in dimension four, using 4ti2, we computed the number of extremal pmfs for different values of $p$, $p=s/100, s=1,\ldots,50$. We find the results in Figure \ref{fig:nr4}.

\begin{figure}
    \centering
    \includegraphics[width=1\linewidth]{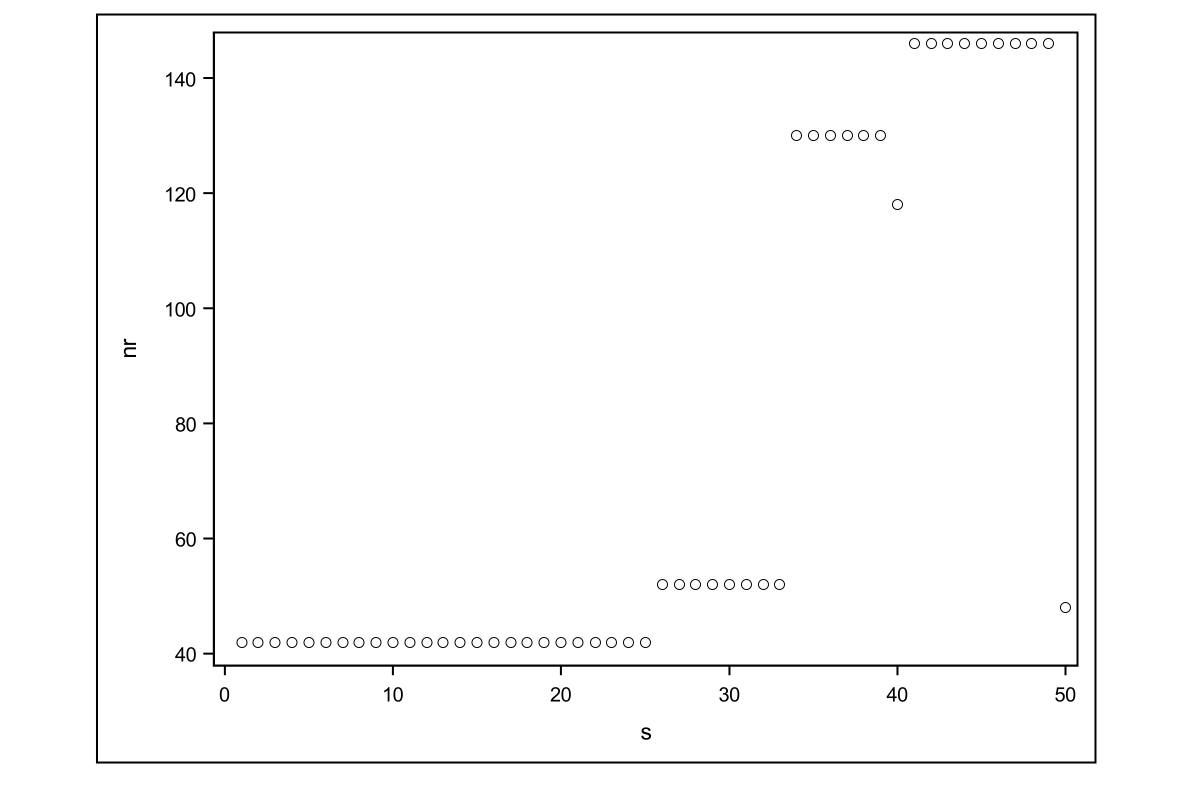}
    \caption{Number of extremal pmfs \emph{vs} $p$, $nr$ is the number of rays and $s=100p$.}
    \label{fig:nr4}
\end{figure}

\section{Acknowledgments}
A previous version of this work has been presented at the IES 2025 Conference, University of Padova, Bressanone-Brixen, Italy, June 25-27, 2025, \cite{IES2025}.


\bibliographystyle{apalike}
\bibliography{biblio.bib}


\end{document}